\documentclass[11pt]{amsart}
\usepackage{graphicx}
\usepackage{latexsym}
\usepackage{amsfonts,amsmath,amssymb}
\newtheorem{theorem}{Theorem}[section]

\newtheorem{lemma}[theorem]{Lemma}

\def\si{\par\smallskip\noindent}
\def\bi{\par\bigskip\noindent}

\def\de{\delta}
\def\eps{\varepsilon}
\def\la{\lambda}
\def\a{\alpha}
\def\be{\beta}

\def\ga{\gamma}
\def\part{\partial}

\usepackage{xcolor}

\newcommand{\beq}{\begin{equation}}
\newcommand{\eeq}{\end{equation}}

\theoremstyle{remark}

\numberwithin{equation}{section}
\date{\today}
\begin{document}

\title[Terminal trees]{Expected number of induced subtrees shared by two independent copies of a random
 tree}

\author{Boris Pittel}
\address{Department of Mathematics, The Ohio State University, Columbus, Ohio 43210, USA}
\email{bgp@math.ohio-state.edu}

\maketitle

\begin{abstract}  Consider a rooted tree $T$ with leaf-set $[n]$, and with all non-leaf vertices
having  out-degree $2$, at least. A rooted tree $\mathcal T$ with leaf-set
$S\subset [n]$ is induced by $S$ in $T$ if $\mathcal T$ is the lowest common ancestor subtree for $S$, with
all its degree-2 vertices suppressed.  
A ``maximum agreement subtree'' (MAST) for a pair 
of two trees $T'$ and $T''$ is a tree $\mathcal T$ with a largest leaf-set $S\subset [n]$ such that $\mathcal T$ is
induced by $S$ both in $T'$ and $T''$. Bryant et al. \cite{BryMcKSte} and Bernstein et al. \cite{Ber} proved, among other results, that 
for $T'$ and $T''$ being two independent copies of a random binary (uniform or Yule-Harding distributed) tree $T$,
the likely magnitude order of $\text{MAST}(T',T'')$ is $O(n^{1/2})$. We prove this bound for a wide class
of random rooted trees : $T$ is a terminal
tree of a branching, Galton--Watson, process with an ordered-offspring distribution of mean $1$, conditioned on 
``total number of leaves is $n$''.

\end{abstract}

\section{Introduction, results} Consider a rooted binary tree $T$, with $n$ leaves  labelled by elements from $[n]$. We visualize this tree with the root on top and the leaves at bottom. Given  $S\subset [n]$, let $v(S)\in V(T)$ denote the lowest common ancestor of leaves in $S$, ($\text{LCA}(S)$.)
Introduce the subtree of $T$ formed by the paths from $v(S)$ to leaves in $S$. Ignoring (suppressing) degree-2 vertices of this subtree (except the root itself), we obtain a rooted binary tree with leaf-set $S$. This binary tree is called ``a tree
induced by $S$ in $T$''.

Finden and Gordon \cite{FinGor} and Gordon \cite{Gor} introduced a notion of a ``maximum agreement subtree'' (MAST)
for a pair 
of such trees $T'$ and $T''$: it is a tree $\mathcal T$ with a largest leaf-set $S\subset [n]$ such that $\mathcal T$ is
induced by $S$ both in $T'$ and $T''$. In a pioneering paper \cite{BryMcKSte}, Bryant, McKenzie and Steel addressed the  problem of a likely order of $\text{MAST}(T'_n,T''_n)$ when $T'_n$ and $T''_n$ are two independent copies of a random binary tree $T_n$. To quote from \cite{BryMcKSte}, such a problem is ``relevant when comparing evolutionary trees for the same set of species that have been constructed from two quite different types of data''.  

It was proved in \cite{BryMcKSte} that $\text{MAST}(T'_n,T''_n)\le (1+o(1)) e2^{-1/2} n^{1/2}$ with probability $1-o(1)$ as $n\to\infty$. The proof was based on a
remarkable property of the uniformly random rooted binary tree, and of few other tree models, known as ``sampling consistency'',
see Aldous \cite{Ald-2}. As observed  by Aldous \cite{Ald}, \cite{Ald2}, sampling consistency makes this model conceptually close to a uniformly random permutation of $[n]$.
Combinatorially, it means that a rooted binary tree $\mathcal T$ with with a leaf-set $S\subset [n]$,
$|S|=s$, is induced by $S$ in exactly $\frac{(2n-3)!!}{(2a-3)!!}$ rooted binary trees with leaf-set $[n]$, regardless of choice of $\mathcal T$. Probabilistically, the rooted binary tree induced by $S$ in $T_n$ is distributed uniformly on the set of all $(2a-3)!!$ such trees. Mike Steel \cite{Steel} pointed out that the sampling consistency of the rooted binary tree follows directly from
a recursive process for generating all the rooted trees in which $S$ induces $\mathcal T$.

Bernstein, Ho, Long, Steel, St. John, and Sullivant \cite{Ber} established a qualitatively similar upper bound 
$O(n^{1/2})$ for the
likely size of a common induced subtree in a harder case of Yule-Harding tree, again relying on sampling consistency of this
tree model. Recently Misra and Sullivant \cite{MisSul} were able to prove the two-sided estimate $\Theta(n^{1/2})$ for the case when
two independent binary trees with $n$ labelled leaves are obtained by selecting independently, and uniformly at random,
two leaf-labelings of the same unlabelled tree. Using the classic results on the length of the longest increasing subsequence  in the uniformly random permutation, Bernstein et al. \cite{Ber} established a first power-law lower bound $cn^{1/8}$
for the likely size of the common induced subtree in the case of the uniform rooted binary tree, and a lower bound 
$c n^{a-o(1)}$, $a=0.344\dots$, for the Yule-Harding model. Very recently, Aldous \cite{Ald} proved that a maximum 
agreement rooted subtree for two independent, uniform, unrooted trees is likely to have  
$n^{\frac{\sqrt{3}-1}{2}-o(1)}\approx n^{0.366}$ leaves, at least. It was mentioned in \cite{Ald} that an upper bound $O(n^{1/2})$
could be obtained by ``the first moment method (calculating the expected number of large common subtrees)''.

In this paper we show that 
the total number of unrooted trees with leaf-set $[n]$, which contains a rooted
subtree induced by $S\subset [n]$, $a=|S|<n$, is $\frac{(2n-5)!!}{(2a-3)!!}$. The proof is based on a two-phase counting
procedure, indirectly inspired by the well-known process of
generating a uniformly random unrooted, leaf-labelled, tree.
It follows  that a rooted binary tree induced by $S$ in the uniformly random unrooted tree on $[n]$ is {\it again\/} distributed uniformly on the set of all $(2a-3)!!$ rooted trees, so that the expected number of agreement trees with $a$ leaves is
$\binom{n}{a}/(2a-3)!!$. { Mike Steel \cite{Steel} informed us recently that this $\frac{(2n-5)!!}{(2a-3)!!}$ 
formula can also be obtained by observing that the number of unrooted binary trees on $[n]$ in which a leaf-set $\mathcal S\subseteq [n]$ induces a given unrooted tree is $\frac{(2n-5)!!}{(2|\mathcal S|-5)!!}$.}
Using the asymptotic estimate from \cite{BryMcKSte}, we have:
a maximum agreement
rooted subtree for two independent copies of the uniformly random unrooted tree is likely to have at most
$(1+o(1)) e2^{-1/2} n^{1/2}$ leaves. 

Our proof of this $\frac{(2n-5)!!}{(2a-3)!!}$ result suggested, strongly, that a bound $O(n^{1/2})$ might, just might, be obtained for a broad class of random rooted trees by using a probabilistic two-phase
counting procedure, where the random tree grows from the root, rather than from leaves.

Consider a Markov branching process initiated by a single progenitor, with a given offspring distribution $\bold p=\{p_j\}_{j\ge 0}$. If
$p_0>0$ and $\sum_j jp_j=1$ (critical case), then the process is almost surely extinct. This process is visualized as a growing rooted tree such that children of each father are ordered, by ``seniority'' say.

Let $T_t$ be the random {\it terminal\/} tree, and let
$T_n$ be $T_t$ conditioned on the event ``$T_t$ has $n$ leaves'', that we label, uniformly at random, by elements of $[n]$. 
For $p_0=p_2=1/2$, $T_n$ is doubly-random, obtained by picking uniformly at random a binary tree with $2n-1$
vertices, such that two children of each father are ordered, and labeling the tree's $n$ leaves  uniformly at random by elements of $[n]$.  This scheme certainly resembles 
a process studied by Harding \cite{Hard} (Section 3.2). A key difference is that Harding considered the case when
the children of a parent are indistinguishable.

In general, we assume that $p_1=0$, 
$\text{g.c.d.}(j:\,p_{j+1}>0)=1$, and that $P(s):=\sum_j p_j s^j$ has convergence radius  $R>1$.
We will show that $\Bbb P_n:=\Bbb P(T_t\text{ has }n\text{ leaves})>0$ for all 
$n$, meaning that $T_n$ is well-defined for all $n$. 

Finally, an out-degree of a vertex in $T_n$ may now exceed $2$. So we add to the definition of a tree $\mathcal T$, induced by $S$ in $T_n$, the condition: the out-degree 
of every vertex from $V(\mathcal T)$ in $T_n$ is the same as its out-degree in $\mathcal T$.

Under the conditions above, we prove that a rooted binary tree $\mathcal T$ with leaf-set $S\subset [n]$, the vertex set $V(\mathcal T)$ and the edge set $E(\mathcal T)$, is induced by $S$ in $T_n$ with probability
\begin{equation}\label{b0}
\begin{aligned}
&\,\,\frac{(n-a)!}{n!\,p_0\,\Bbb P_n}\, \Bbb P(\mathcal T)\, [y^{n-a}] (1-\Phi_1(y))^{-e(\mathcal T)} \Phi'(y),\quad
(e(\mathcal T):=|E(\mathcal T)|),\\
&\Phi(y)=P(\Phi(y)) +p_0(y-1),\quad \Phi_1(y)=\sum_{j>1} jp_j \Phi^{j-1}(y);\\
&\quad\quad\Bbb P(\mathcal T)=\prod_{v\in V(\mathcal T)} p_{d(v,\mathcal T)},
\quad d(v,\mathcal T):= \text{out-degree of }v\text{ in }\mathcal T.
\end{aligned}
\end{equation}
Here $\Phi(y)$ is the probability generating function of the total number of leaves in the terminal tree. We will use \eqref{b0} to show that for $p_0=p_2=1/2$ 
the expected number of
agreement trees with $a$ leaves is $\binom{n}{a}/\bigl[ 2^{a-1}(2a-3)!!\bigr]$. Consequently a maximum--agreement
rooted subtree for two independent copies of the terminal Galton-Watson tree with $n$ leaves, labelled uniformly
at random, is likely to have at most
$(1+o(1)) (e/2) n^{1/2}$ leaves.

Note that in general, because of the factor $\Bbb P(\mathcal T)$, and $e(\mathcal T)$, the probability of $\mathcal T$
being induced by $S$ in $T_n$ depends not only on $|S|$, but also on the whole out-degree sequence of $\mathcal T$.

We use the above identity to prove the following claim. 
Let 
\begin{align*}
&\qquad\qquad\qquad\quad\, c(\bold p):=e\,p_0\,\la \Bigl[\max_r\Bigl(r-\sum\nolimits_{j=2}^{\infty}p_j^2r^j\Bigr)\Bigr]^{-1/2},\\
&\,\,\la:=\max(\chi^{-2},\chi^{-1}),\quad \chi:=(2p_0\sigma^2)^{1/2},\quad \sigma^2:=\sum\nolimits_{j=2}^{\infty}
j(j-1)p_j;
\end{align*}
($c(\bold p)=e/2$ for $p_0=p_2=1/2$). Then, for $\eps\in (0,1/2]$, 
with probability $\ge 1-(1-\eps)^{(1+\eps)c(\bold p)n^{1/2}}$,
the largest number of leaves in an induced subtree shared by two independent copies of the conditioned terminal tree
$T_n$ is at most $(1+\eps)c(\bold p) n^{1/2}$.

For a wide ranging exposition of 
combinatorial/probabilistic problems and methods in theory of phylogeny, we refer the reader to the books \cite{SemSte} by 
Semple and Steel, and \cite{Ste} by Steel. The reader may wish to consult B\'ona and Flajolet \cite{BonFla} for a thought-provoking study of algebraic-analytic properties of binary trees from the references above.

\section{Uniform binary trees} 
Consider a rooted binary tree $T$ with leaf-set $[n]$. For a given $S\subset [n]$, there exists
a subtree with leaf-set $S$, which is rooted at  the lowest vertex common to all $|S|$ paths leading away from $S$ toward the root of $T$.  The vertex set of this lowest common ancestor (LCA) tree is the set of all vertices from the paths in question. Ignoring degree-2 vertices of this subtree (except the root itself), we obtain a rooted binary tree $\mathcal T$. This LCA
 subtree has a name ``a tree  induced by $S$ in $T$'', see \cite{Ald}.
 
Let $T'_n, T^{''}_n$ be two independent copies of the uniformly random rooted binary tree with leaf-set $[n]$. Let $X_{n,a}$ denote the random total number of leaf-sets $S\subset [n]$ of cardinality $a$ that induce the same rooted subtree in $T_n'$ and $T^{''}_n$. Bryant et al. \cite{BryMcKSte} proved that
\begin{equation}\label{Ber}
\Bbb E[X_{n,a}]=\frac{\binom{n}{a}}{(2a-3)!!}.
\end{equation}
The proof was based on sampling consistency of the random tree $T_n$, so that $N(\mathcal T)$, the number of rooted trees on $[n]$
in which $S$ induces a given rooted tree $\mathcal T$ on $S$,  is $\frac{(2n-3)!!}{(2a-3)!!}$, thus dependent only on the
leaf-set size.

Following Aldous \cite{Ald} (see Introduction), we consider the case when a binary tree $T$ with leaf-set $[n]$ is {\it unrooted\/}. 
{ Here the definition of a rooted subtree induced by a leaf-set $S$ with $|S|=a>1$ remains the same, except
 that it makes sense only for $a<n$.  An induced subtree uniquely exists for any such $S$.
Indeed, a vertex $v$ adjacent to any fixed leaf $\ell^*\in [n]\setminus S$ is joined by a unique path to each leaf in $S$. 
By tracing these $a$ paths toward $v$, we determine their first common vertex $v^*$. 
The subtree formed by the paths from $v^*$ to $S$ is induced by $S$ in $T$. Since $T$ is a tree, a subtree induced by $S$
is unique.}

Let now
$T'_n, T^{''}_n$ be two independent copies of the uniformly random (unrooted)  binary tree with leaf-set $[n]$. Let $X_{n,a}$ denote the random total number of leaf-sets $S\subset [n]$ of cardinality $a$ that induce the same rooted subtree in $T_n'$ and $T^{''}_n$.

\begin{lemma}\label{1} Let $a=|S|<n$. Then 
\begin{equation}\label{0.1}
\Bbb E[X_{n,a}]=\frac{\binom{n}{a}}{(2a-3)!!}.
\end{equation}
Equivalently $\mathcal N(\mathcal T)$, the number of unrooted trees $T$ on $[n]$,
in which $S$ induces a given rooted tree $\mathcal T$ with leaf-set $S$, is $\frac{(2n-5)!!}{(2a-3)!!}$.
\end{lemma}
{ \noindent {\bf Note.\/} So the expectation is the same as for the rooted trees $T$ on $[n]$.  
Compared to a recent combinatorial argument by Steele \cite{Ste},
our longer proof is based on machinery of generating functions. An advantage of our argument is its being a precursor of an
avoidably more 
complicated argument in Section 3. There, a random tree grows from root to leaves, rather
than from leaves to root, as it happens for the classic algorithm, used by Steele: the uniformly random binary tree is generated by attaching labelled leaves to a current tree, a leaf at a time.}  
\begin{proof}

Let us evaluate $\mathcal N(\mathcal T)$.  Consider a generic rooted tree with $a$ leaves. For $\mathcal T$ to be induced by its leaves in $T$ with $n$ leaves, it has to be obtained by ignoring degree-$2$ ({\it non-root\/})
vertices in the LCA subtree for leaf-set $S$. 

The outside (third) neighbors of the ignored vertices are the roots of subtrees with 
some $b$ leaves from the remaining $n-a$ leaves, selected in $\binom{n-a}{b}$ ways. The roots of possible trees, attached to internal points chosen from some of $2(a-1)$ edges of $\mathcal T$, 
can be easily ordered. Introduce 
$F(b,k)$, the total number of ordered forests of $k$ rooted trees with $b$ leaves altogether.  By Lemma 4 of Carter et al. \cite{Car} (for the count of unordered trees), we have
\begin{equation}\label{0.6}
F(b, k)=\frac{k(2b-k-1)!}{(b-k)!\,2^{b-k}}.
\end{equation}
It was indicated in \cite{Ber} that \eqref{0.6} follows from
\begin{equation}\label{0.7}
F(b,k)=b!\cdot[x^b]\, B(x)^k,\quad B(x):= 1-\sqrt{1-2x},
\end{equation}
(Semple and Steel \cite{SemSte}). For the reader's convenience here is a
sketch proof of \eqref{0.7} and \eqref{0.6}. We have
\begin{align*}
F(b,k)&=b!\sum_{t_1+\cdots+t_k=b}\prod_{j\in [k]}\frac{(2t_j-3)!!}{t_j!}=b!\sum_{t_1+\cdots+t_k=b}\prod_{j\in [k]}\frac{1}{t_j 2^{t_j-1}}
\binom{2(t_j-1)}{t_j-1}\\
&=b! [x^b]\biggl[\sum_{t\ge 1}\frac{x^t}{t2^{t-1}}\binom{2(t-1)}{t-1}\biggr]^k=b! [x^b]B(x)^k=\frac{k(2b-k-1)!}{(b-k)!\,2^{b-k}};
\end{align*}
for the last two steps we used Equations (2.5.10),  (2.5.16) in Wilf \cite{Wil}.

Introduce $\mathcal F(b,k)$,
the total number of the ordered forests of $k$ binary trees with roots attached to internal points of $\mathcal T$'s edges,
with $b$ leaves altogether. ($b$ leaves have to be chosen from $[n]\setminus (S\cup\{\ell^*\})$, so $b\le n-a-1$.)  Since the total number of integer compositions (ordered partitions) of $k$ with  $j\le 2(a-1)$ positive parts is 
\[
\binom{k-1}{j-1}\binom{2(a-1)}{j}=\binom{k-1}{j-1}\binom{2(a-1)}{2(a-1)-j},
\]
\eqref{0.6} implies
\begin{multline}\label{0.8}
\mathcal F(b,k)=F(b,k)\sum_{j\le 2(a-1)}\binom{k-1}{j-1}\binom{2(a-1)}{2(a-1)-j}\\
=F(b,k)\binom{k+2a-3}{2a-3}=b!\binom{k+2a-3}{2a-3} [x^b] B(x)^k.
\end{multline}
Now, $\sum_{k\le b} \mathcal F(b,k)$ is the total number of ways to expand the host subtree into a full binary subtree rooted at the lowest common ancestor of the $a$ leaves. 
To evaluate this sum,  first denote $\a=2a-3$, $\be=B(x)$ and write
\begin{align*}
\sum_{k\ge 0}\binom{k+\a}{\a} \be^k&=\sum_{k\ge 0} (-\be)^k \binom{-\a-1}{k}=(1-\be)^{-\a-1}.
\end{align*}
Therefore
\begin{align*}
\sum_{k\le b}\binom{k+\a}{\a} [x^b] B(x)^k&=[x^b]\sum_{k\ge 0}\binom{k+\a}{\a} B(x)^k=[x^b]\frac{1}{(1-B(x))^{\a+1}}\\
&=[x^b] (1-2x)^{-\frac{\a+1}{2}}.
\end{align*}
We conclude that 
\begin{equation}\label{0.84}
\sum_{k\le b} \mathcal F(b,k)=b!\Bigl.[x^b] (1-2x)^{-\frac{\a+1}{2}}\Bigr|_{\a=2a-3}=b! [x^b] (1-2x)^{-(a-1)}.
\end{equation}
Recall that $b$ leaves were chosen from $[n]\setminus (S\cup\{\ell^*\})$.
If $b=n-a-1$, then attaching the single remaining leaf to the root $v^*$ we get a binary tree $T$ with leaf-set $[n]$. If
$b\le n-a-2$, we view the expanded subtree as a single leaf, and form an unrooted binary tree with 
$1+ (n-a-b)\ge 3$ leaves, in $[2(n-a-b)-3]!!$ ways. 
Therefore 
$\mathcal N(\mathcal T)$ depends on $a$ only, and with $\nu:=n-a-1$, it is given by
\begin{multline*}
\mathcal N(\mathcal T)=\sum_{b\le \nu} \binom{\nu}{b}b!\, [x^b] (1-2x)^{-(a-1)}\bigl(2(\nu-b)-1\bigr)!!\\
=\nu!\sum_{b\le \nu} [x^b] (1-2x)^{-(a-1)} \cdot [x^{\nu-b}] (1-2x)^{-1/2}\\
=\nu![x^{\nu}](1-2x)^{-a+1/2}=\prod_{j=0}^{n-a-2}(2a-1+2j)
=\frac{(2n-5)!!}{(2a-3)!!};
\end{multline*}
in the first line $(-1)!!:=1$, and for the second line we used
\[
\frac{(2k-1)!!}{k!}=\frac{(2k)!}{2^k (k!)^2}=2^{-k} \binom{2k}{k}=[x^k] (1-2x)^{-1/2}.
\]
Consequently
\begin{equation}\label{0.860}
\Bbb E[X_{n,a}]= \binom{n}{a}(2a-3)!!\biggl[\frac{\mathcal N(\mathcal T)}{(2n-5)!!}\biggr]^2
=\frac{\binom{n}{a}}{(2a-3)!!}.
\end{equation}
\end{proof}

\section{Branching Process Framework}
Consider a branching process initiated by a single progenitor. This process is visualized as a growing rooted tree.
The root is the progenitor, connected by edges to each of its immediate descendants (children), that are {\it ordered\/},
say by seniority.
Each of the children becomes the root of the corresponding subtree, so that the children of all these roots are the grandchildren of the progenitor. We obviously get a recursively defined process. It delivers 
a nested sequence of trees, which is 
either infinite, or terminates at a moment when none of the current leaves have children.

The classic Galton-Watson branching process is the case when the number of each member's children {\bf (a)\/} is independent
of those numbers for all members from the preceding and current generations and {\bf (b)\/} has the same distribution $\{p_j\}_{j\ge 0}$, $(\sum_j p_j=1)$. It is well-known that if $p_0>0$ and $\sum_{j\ge 0}jp_j=1$, then the process terminates with probability $1$, Harris \cite{Har}. Let $T_t$ denote the terminal tree. Given a finite rooted tree $T$, we have
\[
\Bbb P(T_t=T)=\prod_{v\in V(T)} p_{d(v,T)},
\]
where $d(v, T)$ is the out-degree of vertex $v\in V(T)$.  $X_t:=|V(T_t)|$, the total population size by the extinction time, has the probability generating function (p.g.f) $F(x):=\Bbb E[x^{X_t}]$, $|x|\le 1$, that satisfies
\begin{equation}\label{b1}
F(x) =x  P(F(x)),\quad P(\xi):=\sum_{j\ge 0} p_j \xi^j,\,\,(|\xi|\le 1).
\end{equation}
Indeed, introducing $F_{\tau}(x)$ the p.g.f. of the total cardinality of the first $\tau$ generations, we have
\[
F_{\tau+1}(x)=x\sum_{j\ge 0} p_j \bigl[F_{\tau}(x)\bigr]^j=x P(F_{\tau}(x)).
\]
So letting $\tau\to\infty$, we obtain \eqref{b1}. In the same vein, consider the pair $(X_t,Y_t)$, where  $Y_t:\bigl|\{v\in V(T_t):\,
d(v, T_t)=0\}\bigr|$
is the total number of leaves (zero out-degree vertices) of the terminal tree. Then denoting $G(x,y)=\Bbb E\bigl[x^{X_t} y^{Y_t}\bigr]$,
$(|x|,\,|y|\le 1)$, we have 
\begin{equation}\label{b2}
G(x,y)=p_0 xy +x\sum_{j\ge 1}p_j \bigl[G(x,y)\bigr]^j=xP(G(x,y))+p_0x(y-1).
\end{equation}
So, with $\Phi(y):=\Bbb E\bigl[y^Y\bigr]=G(1,y)$, we get
\begin{equation}\label{b3}
\Phi(y)=\sum_{j\ge 1} p_j\Phi^j(y)+p_0y=P(\Phi(y))+p_0(y-1).
\end{equation}
Importantly, this identity allows us to deal directly with the leaf set at the extinction moment: {\it $\Bbb P_k:=[y^k]\, \Phi(y)$ is the probability that $T_t$ has $k$ leaves}. In particular, $\Bbb P_1=[y^1]\Phi(y)=p_0>0$. More generally, $\Bbb P_k>0$
for all $k\ge 1$.
meaning that $\Bbb P(T_t\text{ has }k\text{ leaves})>0$ for all $k\ge 1$. Indeed, for $k\ge 2$, we have
\[
\Bbb P_k=\sum_{j\ge 1}p_j\sum_{k_1+\cdots+k_j=k\atop k_1,\dots, k_j\ge 1}\Bbb P_{k_1}\cdots \Bbb P_{k_j};
\]
so the claim follows by easy induction on $k$. Introducing the event $A_k:=\{T_t\text{ has }k\text{ leaves}\}$, we have
$\Bbb P(A_k)=[y^k]\Phi(y)$.

If $p_0=p_2=1/2$, then the branching process is a nested sequence of binary trees. The equation \eqref{b3} yields
\[
\Phi(y)=1-(1-y)^{1/2}=\sum_{n\ge 1}\left(\frac{y}{2}\right)^n\frac{(2n-3)!!}{n!},\quad |y|\le 1;
\]
so $\Bbb P(A_n)=\frac{(2n-3)!!}{2^n n!}>0$. On the event $A_n$, the total number of vertices is $2n-1$, and each of rooted  binary trees with ordered pairs of children
 is a value of the terminal tree with the same probability $(1/2)^{2n-1}$. 
Now,
\[
\frac{\frac{(2n-3)!!}{2^n n!}}{(1/2)^{2n-1}}=\frac{1}{n}\binom{2(n-1)}{n-1}
\]
is the Catalan number $C(n-1)$, which is the total number of rooted binary trees with $n$ leaves, and $n-1$ non-leaves,
each having $2$ ordered children.
Thus, conditionally on $A_n$, the terminal tree is distributed uniformly on the set of these $C(n-1)$ trees.
We do not have such uniformity for a general $\{p_j\}$, of course.

For a general $\{p_j\}$, on the event $A_n$ we label, uniformly at random, the leaves of $T_t$ by elements of $[n]$.
 We  take liberty to use the same notation $T_n$, as in Section 2, for the resulting {\it doubly\/} random, leaf-labelled tree. 

That $T_n$ is again associated with a recursive process is a hopeful sign that we can get a counterpart of  what we proved for the uniformly random, leaf-labelled binary tree, and also extend an analysis to a more general distribution $\{p_j\}$. 
 
We continue to assume that $p_1=0$. The notion of an induced subtree needs to be expanded, since an out-degree of a vertex now may exceed $2$.
Let $\mathcal T$ be a tree with a leaf-set $S\subset [n]$, such that every non-leaf vertex of $\mathcal T$ has at least
two (ordered) children. We say that $S$ induces $\mathcal T$ in a tree $T_n$ provided
that:  {\bf (a)\/} the LCA subtree for $S$ in $T_n$ is $\mathcal T$ if we ignore vertices of total degree $2$ in this LCA
subtree; 
{\bf (b)} the out-degree of every other vertex in the LCA of $S$ in $T_n$ is the same as its out-degree 
in $\mathcal T$. We call this event $A_n(\mathcal T)$. We evaluate $\Bbb P(A_n(\mathcal T))$ in steps.
Let $|S|=a<n$. Given $b\le n-a$,  let $A_n(\mathcal T,b)$ be the event: 

\noindent {\bf (i)\/} $A_n$ holds, i.e. the terminal tree $T_t$ has $n$ leaves; the uniformly random labelling of the leaves of $T_t$,
that results in the random tree $T_n$ is such that:  

\noindent {\bf (ii)\/} some $b$ elements from $[n]\setminus S$ are chosen as leaf labels for
all the complementary extinction subtrees rooted at degree-2 vertices sprinkled on the edges of $\mathcal T$, forming--together
with $\mathcal T$ on leaf-set $S$--an expanded terminal tree
on leaf-set $S\cup\{b\text{ leaves}\}$, of cardinality $a+b$;  

\noindent {\bf (iii)\/} the terminal tree with $n$ labelled leaves is obtained as follows: we build up a terminal tree with leaf-set $[n]\setminus \bigl(S\cup\{b\text{ leaves}\}\bigr)$ plus an extra super-leaf, which is the root  
of the tree built up in  {\bf (ii)\/}, and replace the super-leaf with this tree.

In summary,  a terminal tree $T_n$, compatible with the event $A_n(\mathcal T,b)$, is built up of terminal subtrees 
with a certain number of leaves, each subtree being delivered by a branching process  that starts at the subtree's root.

Clearly $A_n(\mathcal T)$ is the disjoint union of the events $A_n(\mathcal T,b)$. By the very definition, on the event $A_n(\mathcal T,b)$ the leaf-set $S$ certainly induces $\mathcal T$ in $T_n$. 

To evaluate $\Bbb P(A_n(\mathcal T,b))$ we partition  $A_n(\mathcal T,b)$ into disjoint $\binom{n-a}{b}$ events corresponding to all choices to select $b$ elements of $[n]\setminus S$.
 Let $e(\mathcal T)$ be the total number of edges in $\mathcal T$. For each of these choices, on the event $A_n(\mathcal T,b)$ 
we must have some $k\le b$ terminal subtrees whose roots are  some degree-2 vertices, selected from some of $e(\mathcal T)$ edges, with their respective, nonempty,
leaf-sets forming an ordered set partition of the set of $b$ leaves. 
The root of each of these trees has one child down the host edge of $\mathcal T$, and all the remaining children are outside  of edges of $\mathcal T$. The number of those children is $j$ with 
probability $(j+1)p_{j+1}$, and $\{(j+1)p_{j+1}\}_{j\ge 1}$ is a probability distribution, as
\setlength\abovedisplayskip{7pt}
\setlength\belowdisplayskip{7pt}
\[
\sum_{j\ge 1}(j+1)p_{j+1}=\sum_{j\ge 0}jp_j =1.
\]
So the total number of leaves of terminal subtrees rooted at those outside children is $i$ with probability $[y^i] \Phi_1(y)$, where
\begin{equation}\label{b1+}
\Phi_1(y)=\sum_{j\ge 1}(j+1)p_{j+1} \Phi^j(y).
\end{equation}
The probability $[y^i] \Phi_1(y)$ is positive for each $i\ge 1$, since $[y^i]\Phi^j(y)>0$ for each $j\ge 1$.
Therefore 
a given set of $b$ elements of $[n]$ is the leaf-set of these
terminal subtrees with probability
 \begin{multline*}
b! \sum_{j\le k\le b}\binom{k-1}{j-1}\binom{e(\mathcal T)}{j} \sum_{b_1+\cdots+b_k=b}\prod_{t=1}^k\,
 [y^{b_t}] \Phi_1(y)\\
=b!\!\!\sum_{j\le k\le b}\binom{k-1}{j-1}\binom{e(\mathcal T)}{e(\mathcal T)-j} [y^b]\Phi_1^k(y)=b![y^b]\sum_k\binom{k+e(\mathcal T)-1}{k}\Phi_1^k(y)\\
=b![y^b](1-\Phi_1(y))^{-e(\mathcal T)}.
\end{multline*}
Explanation: $k$ is a generic total number of trees rooted at ordered internal points of some $j$ edges of $\mathcal T$;  $b_t$ is a generic
number of leaves of a $t$-th tree; the product of two
binomial coefficients in the top line is the number of ways to pick $j$ edges of $\mathcal T$ and to select an ordered, $j$-long, composition of $k$; the sum is the probability that the $k$ trees have $b$ leaves in total. $b!$ accounts for the number of ways to assign the chosen $b$ elements as labels of $b$ leaves.

With these complementary trees attached, we obtain a terminal tree with $a+b$ leaves, rooted at the root of $\mathcal T$. 
We denote this expanded tree $\mathcal E(\mathcal T)$.  For $b=n-a$, this is our terminal tree $T$ with $n$ labelled leaves. If 
$b<n-a$, then $\mathcal E(\mathcal T)$ is a subtree of $T$. Specifically, a branching process tree, grown from a progenitor root, terminates when there is a single active leave, meaning that this leaf is about to produce children in accordance with
offspring distribution $\bold p$, and that all other leaves are childless. This single leaf becomes the root of  $\mathcal E(\mathcal T)$, which completes construction of $T$. Now, $p_0$ {\/times\/} the probability that at this moment the number of childless leaves is $n-a-b$ equals $(n-a-b+1)! [y^{n-a-b+1}]\Phi(y)$, which is the probability
of a terminal tree with $n-a-b+1$ leaves labelled by remaining $n-a-b$ elements of $[n]$ plus $1$, accounting for the root of 
$\mathcal E(\mathcal T)$.

Therefore \begin{equation}\label{b4.9}
\begin{aligned}
\Bbb P\bigl(A_n(\mathcal T,b)\bigr)&=\frac{\Bbb P(\mathcal T)}{p_0\,n!}\binom{n-a}{b}\,\times b!\,[y^b](1-\Phi_1(y))^{-e(\mathcal T)}\\
&\quad\times (n-a-b+1)!\, [y^{n-a-b+1}] \Phi(y).
\end{aligned}
\end{equation}
Here $\Bbb P(\mathcal T):=\prod_{v\in V(\mathcal T)} p_{d(v,\mathcal T)}$,
where $\{d(v,\mathcal T)\}$ is the out-degree sequence of 
vertices in $\mathcal T$, that includes the actual out-degree of $\mathcal T$'s root.
Using $(j+1) [y^{j+1}] \Phi(y)=[y^j]\Phi'(y)$, we simplify \eqref{b4.9}:
\[
\Bbb P\bigl(A_n(\mathcal T,b)\bigr)=\frac{ (n-a)!\,\Bbb P(\mathcal T)}{p_0 n!}\,[y^b](1-\Phi_1(y))^{-e(\mathcal T)}\,
[y^{n-a-b}] \Phi'(y).
\]
Summing the last equation for $0\le b\le n-a$, we obtain
\begin{equation}\label{b5}
\Bbb P\bigl(A_n(\mathcal T)\bigr)=\frac{\Bbb P(\mathcal T) (n-a)!}{p_0\,n!}\,
[y^{n-a}] (1-\Phi_1(y))^{-e(\mathcal T)} \Phi'(y).
\end{equation}
 
For $p_0,\,p_2=1/2$,  we have 
$
\Bbb P(\mathcal T)=(1/2)^{2a-1},\Phi_1(y)=\Phi(y), 
$
see \eqref{b1+}. So
\begin{multline*}
[y^{n-a}] (1-\Phi(y))^{-e(\mathcal T)} \Phi(y)=[y^{n-a}] (1-\Phi(y))^{-2a+2}\Phi'(y)\\
=[y^{n-a}] (1-y)^{-a+1}\cdot \frac{1}{2}(1-y)^{-1/2}=\frac{1}{2} [y^{n-a}] (1-y)^{-a+1/2}\\
=2^{-n+a-1}\frac{(2n-3)!!}{(n-a)!\, (2a-3)!!},
\end{multline*}
and, by \eqref{b5}, we have
\begin{equation}\label{b6}
\Bbb P\bigl(A_n(\mathcal T)\bigr)=\frac{(2n-3)!!}{2^{n+a-1}\, n!\,(2a-3)!!}.
\end{equation}
Since $\Bbb P(A_n)=\frac{(2n-3)!!}{2^n n!}$, we conclude that
\begin{equation}\label{b6.1}
\Bbb P(A_n(\mathcal T)|\,A_n)=\frac{\Bbb P\bigl(A_n(\mathcal T)\bigr)}{\Bbb P(A_n)}=\frac{1}{2^{a-1}(2a-3)!!},
\end{equation}
for {\it every\/} binary tree $\mathcal T$ with leaf-set $S\subset [n]$, $|S|=a$. The LHS is the probability that $S$ induces 
$\mathcal T$ in the uniformly random binary tree $T_n$.
\vspace{-1.5mm}
\begin{theorem}\label{newlem} Let $X_{n,a}$ denote the total number of leaf-sets $S\subset [n]$ that induce the same
binary tree in two independent copies of the random binary tree $T_n$.
We have
\begin{equation*}
\Bbb E[X_{n,a}]=\binom{n}{a} [2^{a-1}(2a-3)!!]^{-1};
\end{equation*}
in particular, 
\[
\Bbb E[X_{n,1}]=n,\quad \Bbb E[X_{n,2}]=n(n-1)/4,\quad \Bbb E[X_{n,3}]=n(n-1)(n-2)/72.
\]
Consequently a maximum--agreement
rooted subtree for two independent copies of the terminal Galton-Watson tree with $n$ leaves, labelled uniformly
at random, is likely to have at most
$(1+o(1)) (e/2) n^{1/2}$ leaves.
\end{theorem}
\begin{proof}
The total number of the binary trees with $a$ leaves in question  is $a!\,C(a-1)=2^{a-1}(2a-3)!!$.
Therefore, by \eqref{b6.1},
\begin{equation*}
\Bbb E[X_{n,a}]=\binom{n}{a}\bigl[2^{a-1}(2a-3)!!\bigr]^{-2}\cdot a!\, C(a-1)=\binom{n}{a}\,[2^{a-1}(2a-3)!!]^{-1}.
\end{equation*}
\end{proof}
For a general distribution $\{p_j\}$, we have proved
\begin{lemma}\label{lem3} 
Consider  the branching process with the immediate offspring distribution $\{p_j\}$, such that $p_0>0$, $p_1=0$, and 
$\sum_{j\ge 0}j p_j=1$. With probability $1$, the process eventually stops, so that a finite terminal tree $T_t$ is a.s. well-defined. {\bf (1)}  Setting
$A_n=\{T_t\text{ has }n\text{ leaves}\}$, we have 
\[
\Bbb P (A_n)= [y^n] \Phi(y), \quad \Phi(y)= P(\Phi(y))+p_0(y-1),
\]
where $ P(\eta):=\sum_{j\ge 0} \eta^j p_j$.  {\bf (2)\/} On the event $A_n$, we define $T_n$ as the tree $T_t$
with leaves labelled uniformly at random by elements from $[n]$. Given a rooted tree $\mathcal T$ with leaf-set $S\subset [n]$, set $A_n(\mathcal T)$:= ``$A_n$ \text{holds; }$S$\text{ is a subset of }$T_t's$\linebreak\text{ leaf-set }; $S$\text{ induces }$\mathcal T$ in $T_n$''. Then
$\Bbb P(A_n(\mathcal T)|\,A_n)=\frac{\Bbb P(A_n(\mathcal T))}{\Bbb P (A_n)}$, where
\begin{equation}\label{b7}
\begin{aligned}
&\,\Bbb P\bigl(A_n(\mathcal T)\bigr)=\frac{\Bbb P(\mathcal T) (n-a)!}{p_0\,n!}
\, [y^{n-a}] (1-\Phi_1(y))^{-e(\mathcal T)} \Phi'(y),\\
&\,\,\Phi(y)= P(\Phi(y)) +p_0(y-1),\quad \Phi_1(y)=\sum_{j>1}jp_j \Phi^{j-1}(y),\\
&\qquad\qquad\qquad\Bbb P(\mathcal T)=\!\!\!\prod_{v\in V(\mathcal T)}\!\! p_{d(v,\mathcal T)};
\end{aligned}
\end{equation}
$V(\mathcal T)$ and $e(\mathcal T)$ are the set of vertices and the number of edges of $\mathcal T$.
\end{lemma}
{\bf Note.\/} For $p_0=p_2=1/2$, $\Bbb P\bigl(A_n(\mathcal T)\bigr)$ turned out to be dependent only on the number of leaves of $\mathcal T$. The formula \eqref{b7} clearly shows that, in general, this probability depends
on shape of $\mathcal T$. Fortunately, this dependence is confined to a single factor $\Bbb P(\mathcal T)$, since the rest depends on
two scalars, $a$ and $e(\mathcal T)$. Importantly, these parameters are of the same order of magnitude. Indeed, if 
if $a>1$ and $\Bbb P(\mathcal T)>0$ then $e(\mathcal T)\ge \max\bigl(a, 2 |V_{\text{int}}(\mathcal T)|\bigr)$, where  $V_{\text{int}}(\mathcal T)$ is the set of non-leaf vertices of $\mathcal T$. Hence
\[
a\le e(\mathcal T)=|V(\mathcal T)|-1=|V_{\text{int}}(\mathcal T)|+a-1\le \frac{e(\mathcal T)}{2}+a-1,
\]
so that
\begin{equation}\label{b8}
a\le e(\mathcal T)\le 2(a-1).
\end{equation}

\subsection{Asymptotics} From now on we assume that 
the series $\sum_j p_j s^j$ has convergence radius $R>1$. 
\begin{lemma}\label{lem4} Suppose that $d:=\text{g.c.d.}\{j\ge 1:\,p_{j+1}>0\}=1$. Let $\sigma^2:=\sum_{j\ge 0}j(j-1) p_j$, i.e. $\sigma^2$ is the variance of the immediate offspring,
since $\sum_{j\ge 0} jp_j=1$.  Then
\[
\Bbb P(A_n)=\frac{(2p_0)^{1/2}}{\sigma}\frac{(2n-3)!!}{2^n\,n!}+ O(n^{-2})=\biggl(\frac{p_0}{2\pi \sigma^2}\biggr)^{1/2}
\!\!\!n^{-3/2} +O(n^{-2}).
\]
\end{lemma}

\begin{proof} According to Lemma \ref{lem3}, we need to determine an asymptotic behavior of the coefficient in the power series $\Phi(z)=\sum_{n\ge 1}z^n \Bbb P(A_n)$, where $\Phi(z)$ is given implicitly by the functional equation
$\Phi(z)=\sum_{j\ge 0}p_j \Phi^j(z)+p_0(z-1),\,(|z|\le 1)$.  

In 1974 Bender \cite{Ben} sketched a proof of the following general claim.
\begin{theorem}\label{bender1} Assume that the power series $w(z)=\sum_n a_nz^n$ with nonnegative coefficients satisfies $F(z,w)\equiv 0$. Suppose that there exist  $r>0$ and $s>0$ such that: {\bf (i)\/} for some $R>r$ and $S>s$, $F(z,w)$ is analytic for $|z|<R$ and $w< S$; {\bf (ii)\/} $F(r,s)=F_w(r,s)=0$; {\bf (iii)\/} $F_z(r,s)\neq 0$ and $F_{ww}(r,s)\neq 0$;
{\bf (iv)\/} if $|z|\le r$, $|w|\le s$, and $F(z,w)=F_w(z,w)=0$, then $z=r$ and $w=s$. Then
$ a_n\sim \bigl((rF_z(r,s))/(2\pi F_{ww}(r,s))\bigr)^{1/2} n^{-3/2} r^{-n}$.
\end{theorem}
\noindent The remainder term aside, that's exactly what we claim in Lemma \ref{lem4} for our $\Phi(z)$. The proof in \cite{Ben} relied on an appealing conjecture that, under the conditions {\bf (i)\/}-{\bf (iv)\/}, $r$ is the radius of convergence for the power series for $w(z)$,
and $z=r$ is the only singularity for $w(z)$ on the circle $|z|=r$. However, ten years later Canfield \cite{Can} found an
example of $F(z,w)$ meeting the four conditions in which $r$ and the radius of convergence for $w(z)$ are not the same.
Later Meir and Moon  found some conditions sufficient for validity of the conjecture.
Our
equation $\Phi(z)=P(\Phi(z))+p_0(z-1)$ is a special case of $w=\phi(w)+h(z)$ considered in \cite{MeiMoo}. For the  
conditions from \cite{MeiMoo} to work in our case, it would have been necessary to have  $\big|P'(w)/P(w)\big|\le 1$ for complex
$w$ with $|w|\le 1$, a strong condition difficult to check. (An interesting discussion of these issues can be found 
in an encyclopedic book by Flajolet and Sedgewick \cite{FlaSed} and an authoritative survey by Odlyzko \cite{Odl}.)

Let $r$ be the convergence radius for the powers series representing $\Phi(z)$; so that $r\ge 1$, since $\Bbb P(A_n)\le 1$. By implicit differentiation,
we have
\[
\lim_{x\uparrow 1}\Phi'(x)=\lim_{x\uparrow 1}\frac{p_0}{1-\Bbb P_w(\Phi(x))}=\infty,
\]
since $\lim_{x\uparrow 1}P_w(\Phi(x))=\sum_j jp_j=1$. Therefore $r=1$. Turn to complex $z$.
For $|z|<1$, we have $F(z, \Phi(z))=0$, where $F(z,w):=p_0(z-1)+P(w)-w$ is analytic as a function of $z$ and $w$ subject to $|w|<1$. Observe that $F_w(z,w)= P'(w)-1=0$ is possible only if $|w|\ge 1$, since for $|w|<1$ we have
$|P'(w)|\le P'(|w|)< P'(1)=1$. If $|w|=1$ then $P'(w)=\sum_{j\ge 2} jp_j w^{j-1}=1$  if and only if $w=w_k:=\exp\Bigl(i\frac{2\pi k}{d}\Bigr)$, and $k=1,\dots,d$. Notice also
that 
\begin{equation}\label{b9}
 P(w_k)=\sum_{j\ge 0}p_j w_k^j=p_0+w_k\sum_{j\ge 2}p_j w_k^{j-1}=p_0+w_k(1-p_0).
\end{equation}
Now, $z$ is a singular point of $\Phi(z)$ if and only if $|z|=1$ 
and $(P'(w)\!-1\!)\big|_{w=\Phi(z)}\!=\!0$, i.e. if and only if $\Phi(z)=w_k$ for some $k\in [d]$, which is equivalent to 
\[
p_0(z-1)+P(w_k)-w_k=0.
\]
Combination of this condition with \eqref{b9} yields $z=w_k$. Therefore the set of all singular points of $\Phi(z)$ on the circle $|z|=1$ is the set of all $w_k$ such that $\Phi(w_k)=w_k$. Notice that
\[
P''(w_k)=w_k^{-1}\sum_{j\ge2}j(j-1)p_j w_k^{j-1}=w_k^{-1}\sum_{j\ge 2}j(j-1)p_j=w_k^{-1}\sigma^2\neq 0.
\]
So none of $w_k$ is an accumulation point of roots of $P'(w)-1$ outside the circle $|w|=1$, i.e. $\{w_k\}_{k\in [d]}$
is the full root set of  $P'(w)-1$ inside the circle $|w|=1+\rho_0$, for some small $\rho_0>0$.

Consequently, if $d=1$, then $z=1$ is the only singular point of $\Phi(z)$ on the circle $|z|=1$. Define the argument $\text{arg}(z)$ by the condition $\text{arg}(z)\in [0,2\pi)$.
By the analytic implicit function theorem applied to $F(z,w)$, for every $z_{\a}=e^{i\a}$, $\eps\le \a \le 2\pi- \eps$, a small $\eps>0$
being fixed, there exists an analytic function 
$\Psi_{\a}(z)$ defined on $D_{z_{\a}}(\rho)$--an open disc centered at $z_{\a}$, of a radius $\rho=\rho(\eps)<\rho_0$ small enough
to make $\eps/2\le \text{arg}(z)\le 2\pi-\eps/2$ for all $z\in D_{z_{\a}}$-- such that 
$F(z,\Psi_{\a}(z))=0$ for $z\in D_{z_{\a}}$ and $\Psi_{\a}(z)=\Phi(z)$ for $z\in D_{z_{\a}}$ with $|z|\le 1$. Together, these local
analytic continuations determine an analytic continuation of $\Phi(z)$ to a function $\hat{\Psi}(z)$ determined, and bounded, for $z$ with $|z|<1+\rho$, $\text{arg}(z)\in[\eps, 2\pi-\eps]$.

Since $z_0=1$ is the singular point of $\Phi(z)$, there is no analytic continuation of $\Phi(z)$ for $|z|>1$ and $|z-1|$ small.
So instead we delete an interval $[1,1+\rho)$ and continue $\Phi(z)$ analytically into
the remaining part of a disc centered at $1$. Here is how. We have $F_{ww}(1,\Phi(1))=P''(1)=\sum_j j(j-1)p_j =\sigma^2>0$.
By a ``preparation'' theorem due to Weierstrass, (Ebeling \cite{Ebe}, Krantz \cite{Kra}), already used by Bender \cite{Ben} for the same purpose in the general setting,  
there exist two open discs $D_1$ and $\mathcal D_1$ such that for $z\in D_1$ and $w\in \mathcal D_1$ we have
\[
F(z,w)= \bigl[(w-1)^2+c_1(z)(w-1)+c_2(z)\bigr] g(z,w),
\]
where $c_j(z)$ are analytic on $D_1$, $c_j(1)=0$, and $g(z,w)$ is analytic, non-vanishing, on $D_1\times \mathcal D_1$. (The degree $2$ of the 
polynomial is exactly the order of the first non-vanishing derivative of $F(z,w)$ with respect to $w$ at $(1,1)$.) So for $z\in D_1$ and $w\in \mathcal D_1$, $F(z,w)=0$ is equivalent to 
\[
(w-1)^2+c_1(z) (w-1)+c_2(z)=0\Longrightarrow w=1+(z-1)^{1/2} h(z),
\]
where $h(z)$ is analytic at $z=1$.
Plugging the power series $w=1+(z-1)^{1/2}h(z)=1+(z-1)^{1/2}\sum_{j\ge 0} h_j(z-1)^j$ into equation $F(z,w)=0$, and expanding $P(w)$ in
powers of $w-1$, we can compute the coefficients $h_j$.
In particular,
\[
w(z)=1-\ga(1-z)^{1/2}+ O(|z-1|),\quad \ga:=(2p_0)^{1/2}\sigma^{-1}.
\]
For $z$ real and $z\in (0,1)$, we have $\Phi(z)=1-\ga (1-z)^{1/2}+O(1-z)$. So to use $w(z)$ as an extension $\tilde{\Psi}(z)$ we need to choose $\sqrt{\xi}=|\xi|^{1/2} \exp\bigl(i \text{Arg}(\xi)/2\bigr)$, where $\text{Arg}(\xi)\in (-\pi,\pi)$.

The continuations $\hat{\Psi}(z)$ and $\tilde{\Psi}(z)$ together determine an analytic continuation of $\Phi(z)$ into a function $\Psi(z)$ which is analytic and bounded on a disc $D^*=D_0(1+\rho^*)$
{\it minus a cut\/} $[1, 1+\rho^*)$, $\rho^*>0$ being chosen sufficiently small, such that
\[
\Psi(z)\underset {z\in D^*\setminus [1,1+\rho^*)\atop z\to 1}=1-\ga (1-z)^{1/2}+ O(|z-1|).
\]
It follows that
\begin{multline*}
\Bbb P(A_n)=\frac{1}{2\pi i}\oint_{|z|=1}\frac{\Phi(z)}{z^{n+1}}\,dz\\
=\frac{1}{2\pi}\oint_{|z|=1+\rho*}\frac{\Psi(z)}{z^{n+1}}+\frac{1}{2\pi i}\int_1^{1+\rho*} \frac{2i \ga(x-1)^{1/2}+O(x-1)}{x^{n+1}}\,dx\\
=O\bigl((1+\rho^*)^{-n}+n^{-2}\bigr)+\frac{\ga}{\pi}\int_1^{\infty}\frac{(x-1)^{1/2}}{x^{n+1}}\,dx.
\end{multline*}
For the second line we integrated $\Psi(z)/z^{n+1}$ along the {\it limit\/} contour :
it consists of the directed circular arc $z =(1+\rho^*) e^{i\a}$, $0<\a< 2\pi$ and a detour part formed by two opposite-directed line segments, one from $z=(1+\rho^*)e^{i(2\pi-0)}$ to $z=e^{i(2\pi-0)}$ and another from $z=e^{i(+0)}$ to $z=(1+\rho^*)e^{i(+0)}$. By the formula 3.191(2)
from Gradshteyn and Ryzik \cite{GraRyz}, we have
\setlength\abovedisplayskip{3pt}
\setlength\belowdisplayskip{3pt}
\begin{align*}
&\int_1^{\infty}\frac{(x-1)^{1/2}}{x^{n+1}}\,dx=B(n-1/2,3/2)=\frac{\Gamma(n-1/2)\, \Gamma(3/2)}{\Gamma(n+1)}\\
&=\frac{\prod_{m=2}^n\Bigl(n-\frac{2m-1}{2}\Bigr)}{n!}\cdot \frac{1}{2}\Gamma^2(1/2)=\pi\frac{(2n-3)!!}{2^n\, n!}.
\end{align*}
Therefore
\[
\frac{\ga}{\pi}\int_1^{\infty}\frac{(x-1)^{1/2}}{x^{n+1}}\,dx=\ga\frac{(2n-3)!!}{2^n\, n!}=\frac{\ga}{2\pi^{1/2}n^{3/2}}+O(n^{-2}).
\]
Recalling that $\ga=(2p_0)^{1/2}\sigma^{-1}$, we complete the proof of the Lemma.
\end{proof}
 \vspace{-1mm}
Using  Lemma \ref{lem3} and Lemma \ref{lem4} we prove
\begin{theorem}\label{thm4} Suppose that $\bold  p=\{p_j\}$ is such that $p_0>0$, $p_1=0$ and $\text{g.c.d.}(j\ge 1: p_{j+1}>0)=1$. {\bf (i)\/} Then $T_n$, the random finite terminal tree of 
of Galton-Watson process with offspring distribution $\bold p=\{p_j\}$,  and $n$ leaves,  labelled 
uniformly at random by elements of $[n]$, is well defined for every $n$. {\bf (ii)\/} 
Let $X_{n,a}$ be the total number of subsets $S\subset [n]$, $|S|=a$, such that $S$ induces the same subtree  in two
independent copies of $T_n$. Then 
there is an explicit constant $c(\bold p)$, such that:
for $\eps\in (0,1/2]$ and $a\ge (1+\eps)c(\bold p) n^{1/2}$, we have $\Bbb E[X_{n,a}]\le (1-\eps)^{a}$.
So, with probability $\ge 1-(1-\eps)^{(1+\eps)c(\bold p)n^{1/2}}$,
the largest number of leaves in a common induced subtree is $(1+\eps)c(\bold p) n^{1/2}$, at most. ($c(\bold p)\!=\!e/2$ if
$p_0=p_2=1/2$.)
\end{theorem}
\begin{proof} 
By Lemma \ref{lem3}, 
\setlength\abovedisplayskip{5pt}
\setlength\belowdisplayskip{5pt}
\begin{equation}\label{b10}
\begin{aligned}
&\qquad\qquad\qquad\qquad \Bbb P(S\text{ induces }\mathcal T\text{ in }T_n)\\
&=\frac{(n-a)!}{\Bbb P(A_n)p_0\,n!}\,\left(\prod_{v\in V(\mathcal T)} p_{d(v,\mathcal T)}\right) [y^{n-a}] (1-\Phi_1(y))^{-e(\mathcal T)} \Phi'(y),\\
&\qquad\Phi(y)=\sum_{j>1}p_j \Phi^j(y) +p_0 y,\quad \Phi_1(y)=\sum_{j>1} jp_j \Phi^{j-1}(y).
\end{aligned}
\end{equation}
$\Phi(y)$ and $(1-\Phi_1(y))^{-e(\mathcal T)}$ are generating functions with positive coefficients; then so is their product.
So we obtain a Chernoff-type bound: for $r\in (0,1)$,
\begin{equation}\label{b11}
[y^{n-a}] (1-\Phi_1(y))^{-e(\mathcal T)} \Phi'(y)\le \frac{(1-\Phi_1(r))^{-e(\mathcal T)} \Phi'(r)}{r^{n-a}}.
\end{equation}
As $\Phi(r)=1-\ga(1-r)^{1/2}+O(1-r)$, we have
\begin{align*}
1-\Phi_1(y)&=1-\sum_{j>1}jp_j\Phi^{j-1}(y)=1-P'(\Phi)\\
&=1-P^{\prime}(1)-P^{\prime\prime}(1)(\Phi-1)+O((\Phi-1)^2)\\
&=\chi (1-r)^{1/2}\bigl(1+O((1-r)^{1/2})\bigr),\quad \chi:=\ga\sigma^2=(2p_0\sigma^2)^{1/2},
\end{align*}
and $\Phi'(y)=O((1-r)^{-1/2})$. So the RHS of \eqref{b11} is of order $\chi^{-e(\mathcal T)}f(r)$, with
$
f(r):= (1-r)^{-e(\mathcal T)/2}r^{-n+a}.
$
$f(r)$ attains its maximum 
\[
 \frac{\bigl[n-a+e(\mathcal T)/2\bigr]^{n-a+e(\mathcal T)/2}}{(n-a)^{n-a}\, \bigl[e(\mathcal T)/2\bigr]^{e(\mathcal T)/2}}
 \le c n^{1/2} \binom{n-a+e(\mathcal T)/2}{n-a}
\]
at $r_n=\frac{n-a}{n-a+e(\mathcal T)/2}$, which is $1 - \Theta(a/n)$, since $e(\mathcal T)\in [a,2a]$ and $a=o(n)$, 
see \eqref{b8}. In
addition, the binomial factor is at most $\binom{n}{n-a}=\binom{n}{a}$. Hence
\begin{align*}
&[y^{n-a}] (1-\Phi_1(y))^{-e(\mathcal T)} \Phi'(y)
\le c_1n\Bigl(\chi+O(\sqrt{a/n})\Bigr)^{-e(\mathcal T)}\binom{n}{a}\\
&\le c_1 n \la^a \binom{n}{a},\quad
\la>\la(\bold p)=\left\{\begin{aligned}
&\chi^{-2},&&\chi\le 1,\\
&\chi^{-1},&&\chi> 1.\end{aligned}\right.
\end{align*}
Hence, using the second line equation in \eqref{b10}, 
and $\Bbb P(A_n)=\Theta(n^{-3/2})$, we obtain
\begin{equation}\label{b12}
\Bbb P(S\text{ induces }\mathcal T\text{ in }T_n)\le \frac{c_1n}{a!}\,\la^a \prod_{v\in V(\mathcal T)} p_{d(v,\mathcal T)}.
\end{equation}
Therefore  we have:
\begin{equation*}
\Bbb E[X_{n,a}]
\le (n)_a \Bigl(\frac{c_1n^{5/2}\la^a}{a!}\Bigr)^2\sum_{\mathcal T}\prod_{v\in V(\mathcal T)} p^2_{d(v,\mathcal T)},
\end{equation*}
where the last sum is over all (finite) rooted trees $\mathcal T$, with ordered children and $a$ unlabelled leaves. 
Define the probability distribution $\bold q=\{q_j\}$: $q_0=1-\sum_{j\ge 2}p_j^2>0$, $q_1=0$, $q_j=p_j^2$ for $j\ge 2$.
Then we have
\begin{equation*}
\prod_{v\in V(\mathcal T)}\!\!p^2_{d(v,\mathcal T)}=\left(\frac{p^2_0}{q_0}\right)^a\prod_{v\in V(\mathcal T)}q_{d(v,\mathcal T)}.
\end{equation*}
Observe that $\sum_{j\ge 2}jq_j< \sum_{j\ge 2}jp_j=1$. Therefore the process with the offspring distribution $\bold q$
is almost surely extinct, implying that
\begin{align*}
&\sum_{\mathcal T}\prod_{v\in V(\mathcal T)} p^2_{d(v,\mathcal T)}
=\left(\frac{p^2_0}{q_0}\right)^a\sum_{\mathcal T}\prod_{v\in V(\mathcal T)}q_{d(v,\mathcal T)}\le \left(\frac{p^2_0}{q_0}\right)^a
[y^a]\Psi(y),\\
&\qquad\qquad\qquad\Psi(y)=\sum_{j\ge 0} q_j \Psi^j(y) +q_0(y-1),
\end{align*}
i.e. $\Psi(y)$ is the p.g.f. of the number of leaves in the terminal tree for the distribution $\bold q$. Let $C$ be a contour
around $y=0$ within the circle $|y|<1$ and let $\mathcal C$ be a circle of radius $\rho$ which is the {\it maximum\/} point
of $r-\sum_{j\ge 2}q_j r^j$. Then, using $y=q_0^{-1}\bigl(\Psi(y)-\sum_{j\ge 2}q_j\Psi^j(y)\bigr)$, we have
\begin{multline*}
[y^a]\Psi(y)=a^{-1} [y^{a-1}]\Psi'(y)=\frac{1}{2\pi i a}\oint_{C}\frac{\Psi'(y)}{y^a}\,dy\\
=\frac{q_0^a}{2\pi i a}\oint_{\mathcal C}\frac{d\eta}{\Bigl(\eta-\sum_{j\ge 2}q_j\eta^j\Bigr)^a}
\le \frac{q_0^a\rho}{a\Bigl(\rho-\sum_{j\ge 2}q_j\rho^j\Bigr)^a}.
\end{multline*}
Therefore
\[
\sum_{\mathcal T}\prod_{v\in V(\mathcal T)} p^2_{d(v,\mathcal T)}\le \frac{\rho\, p_0^{2a}}{a\Bigl(\rho-\sum_{j\ge 2}q_j\rho^j\Bigr)^a}.
\]
So using $a!\ge (a/e)^a$, $\binom{n}{a}\le (ne/a)^a$, we obtain: for all $\la>\la(\bold p)$,
\[
\Bbb E[X_{n,a}]\le c_2\frac{n^5\binom{n}{a}}{a!}\Bigl(\frac{(p_0\la)^2}{\rho-\sum\nolimits_{j\ge 2}q_j\rho^j}\Bigr)^a\le c_3n^5\biggl(\frac{n}{a^2}\,\frac{(ep_0\la)^2}{\rho-\sum\nolimits_{j\ge 2}q_j\rho^j}\biggr)^a
\to 0,
\]
if $a\ge (1+\eps) c(\bold p) n^{1/2}$, $c(\bold p):=ep_0\la(\bold p) \Bigl(\rho-\sum_{j\ge 2}q_j\rho^j\Bigr)^{-1/2}$. Here
$\la(\bold p)=\max(\chi(\bold p)^{-2},\,\chi(\bold p)^{-1})$, and $\chi(\bold p)=(2p_0\sigma^2)^{1/2}$. 
\end{proof} 
{\bf Acknowledgment.\/} I owe a debt of genuine gratitude to Ovidiu Costin and Jeff McNeal for guiding 
me to the Weierstrass separation theorem. I thank Daniel Bernstein, Mike Steel, and Seth Sullivant for an important feedback regarding the references \cite{BryMcKSte}, \cite{Ber} and \cite{Ste}. Mike steered me away from pursuing a false lead
regarding the shapes of binary trees.

\end{document}